\documentclass[12pt]{amsart}
\usepackage[utf8]{inputenc}
\usepackage[margin=1in]{geometry}
\usepackage[english]{babel}

\usepackage{graphicx}
\usepackage{amsmath}
\usepackage{amssymb}
\usepackage{appendix}
\usepackage{amsthm}
\usepackage{booktabs}
\usepackage{graphicx}
\usepackage{subcaption}
\usepackage{comment}
\usepackage{enumerate}

\usepackage[
    backend=biber,
      style=ieee,
    sorting=nyt,
    dashed=false
  ]{biblatex}
\usepackage{hyperref}
\usepackage[hyphenbreaks]{breakurl}
\usepackage{xurl}
\hypersetup{
    colorlinks=true,
    linkcolor=blue,
    citecolor=red,
    breaklinks=true
}

\newtheorem{theorem}{Theorem}[section]
\newtheorem{corollary}[theorem]{Corollary}
\newtheorem{prop}[theorem]{Proposition}
\newtheorem{lemma}[theorem]{Lemma}
\newtheorem{example}[theorem]{Example}

\theoremstyle{definition}

\theoremstyle{remark}
\newtheorem{remark}[theorem]{Remark}

\newtheoremstyle{proofboldstyle}
  {}{}{}{}{\bfseries}{.}{.5em}{{\thmname{#1 }}{\thmnumber{#2}}{\thmnote{#3}}}
\theoremstyle{proofboldstyle}

\usepackage{amsmath}
\usepackage{amssymb}
\usepackage{mathrsfs}
\usepackage{esint}

\newcommand{\R}{{\mathbb R}} 
\newcommand{\Rdd}{\mathbb{R}^{2d}}

\newcommand{\E}{{\mathcal{E}}}
\newcommand{\F}{{\mathcal{F}}}
\newcommand{\Fw}{\mathcal{F}_{\mathrm{W}}}
\newcommand{\Fs}{\mathcal{F}_{\sigma}}
\newcommand{\LtRd}{L^2(\mathbb{R}^d)}

\newcommand{\Rd}{\mathbb{R}^{d}}
\newcommand{\cS}{{\mathcal{S}}}

\newcommand{\Sp}{\mathcal{S}^p}
\newcommand{\tr}{{\text{tr}}}

\usepackage[dvipsnames]{xcolor}
\title[Fourier restriction for operators and functions]{Fourier restriction for Schatten class operators and functions on phase space}
\author{Franz Luef}
\address{Department of Mathematical Sciences,
         Norwegian University of Science and Technology,
         Trondheim, Norway}

\email{franz.luef@ntnu.no}

\author{Helge J. Samuelsen}
\address{Department of Mathematical Sciences,
         Norwegian University of Science and Technology,
         Trondheim, Norway}
\email{helge.j.samuelsen@ntnu.no}
\date{\today}

\begin{document}

\begin{abstract}
We formulate a variant of Fourier restriction for operators in Schatten classes, where the Fourier-Wigner transform of a bounded operator replaces the Fourier transform of a function. The Fourier-Wigner transform is closely related to the group Fourier transform of the Heisenberg group. The first result shows that Fourier-Wigner restriction for Schatten class operators is equivalent to the restriction of the symplectic Fourier transform of functions on phase space. We deduce various Schatten class results for the quantum Fourier extension operator and answer a conjecture by Mishra and Vemuri concerning the Weyl transform of measures in the affirmative. 
\end{abstract}
\maketitle
\section{Introduction}\label{sec:Intro}
Fourier restriction has attracted a lot of attention over the years and has shown to be intimately related to various challenging problems in harmonic analysis, see \cite{Tao}. The question Feffermann and Stein put forward in \cite{Fefferman} was whether restricting the Fourier transform of $L^p$-functions to the unit sphere gives a well-defined bounded operator, the Fourier restriction operator, on the class of $L^p$-spaces. More generally, given a measure on the Euclidean space $\R^n$ for what exponents $p,q\in[1,\infty]$ does the Fourier restriction operator define a bounded operator from $L^p(\R^n)$ to $L^q(\mu)$?

Stein's Fourier restriction conjecture states that the Fourier restriction operator to the unit sphere is a bounded operator from $L^p(\R^{n})$ to $L^q(\mathbb{S}^{n-1})$ for $p$ and $q$ in the range $1\leq p< \frac{2n}{n+1}$ and $1\leq q\leq p'\frac{n-1}{n+1}$.
A celebrated theorem by Tomas and Stein \cite{Schlag,Tomas} shows that the Fourier restriction operator is indeed bounded from $L^p(\R^n)$ to $L^2(\mathbb{S}^{n-1})$ for $p\leq(2n+2)/(n+3)$, and Knapp's example demonstrates that this range of exponents is sharp. The case $q\neq 2$ is not yet fully solved for $n\geq3$. The conjecture in dimension $n=2$ was first proved by Zygmund \cite{Zygmund} for the unit circle, and can also be found in \cite{Demeter,Schlag} for more general curves.

Several techniques have been developed in an attempt to solve the conjecture. Notable examples are the $\varepsilon$-removal trick of Tao \cite{Tao}, multilinear to linear Fourier restriction estimates by Bourgain and Guth \cite{Bourgain_multi_to_lin} (where they utilise the Bennet-Carbery-Tao multilinear restriction estimates \cite{Tao_multi}), and the polynomial method of Guth \cite{Guth_poly1,Guth_poly2}.

Another important technique is that of decoupling, first introduced by Wolff \cite{Wolff_decoup}, and further extended upon by Bourgain and Demeter \cite{Bourgain_Demeter_decoup}. It has seen applications in Fourier restriction, for instance by \L aba and Wang to prove existence of measure with near optimal restriction properties \cite{Laba_Wang}, as well as playing a crucial role in the proof of the Vinogradov mean value theorem \cite{Bourgain_Demeter_Guth}.

Our main motivation is to propose a formulation of the Fourier restriction problem for operators in the Schatten classes $\Sp$ which one might think of as a restriction problem for the Fourier-Wigner transform \cite{Folland} or the Fourier-Weyl transform \cite{Werner}. The Fourier-Wigner transform is closely related to the group Fourier transform $\Fw$ of the Heisenberg group. A key observation is the equivalence of this operator variant of the Fourier restriction problem with the one of the symplectic Fourier transform $\Fs$ for functions on phase space $\Rdd$. 

The framework is based on the representation theory of the Heisenberg group, in particular, the Schr\"odinger representation of phase space $\Rdd$ and different quantization schemes, such as Weyl quantization and Berezin quantization or $\tau$-quantization. In \cite{Werner} this setting was used to build a theory of harmonic analysis for operators, called quantum harmonic analysis, and the methods and tools from quantum harmonic analysis are at the core of all our arguments. Recent developments in time-frequency analysis have linked quantum harmonic analysis with the theory of localization operators, time-frequency representations, and have led to the introduction of notions such as the Cohen class of an operator. The latter is essential in our approach.

Our definition of the Fourier-Wigner restriction problem for Schatten class operators is based on the representation theory of the Heisenberg group and one might want to notice that the inverse group Fourier transform of the Heisenberg group gives a whole family of Fourier-Wigner transforms, which are all unitarily equivalent after using a suitable dilation of the Heisenberg group. We discuss this in more detail in Section \ref{se:background}. Thus, one might also recast our results as restriction theorems for the Heisenberg group. We will phrase our results for one ``fiber" of the group Fourier transform of the Heisenberg group since treating it in full generality does not add to the problem at hand and the statements become more transparent. 

The main takeaway of our paper is the equivalence of the Fourier restriction theorem on phase space of functions with the Fourier-Wigner restriction of operators. The equivalence between statements in classical harmonic analysis and those in quantum harmonic analysis has also been noticed by Fulsche and Rodriguez-Rodriquez for spectral synthesis \cite{Fulsche}, and for $\ell^p$-decoupling theorems by the second author \cite{Samuelsen}.

Fourier restriction for operators has been addressed in two recent publications \cite{Hong,Mishra}. There is some overlap between the results of these contributions to the subject and those put forward here. Still, the methods underlying the approaches in \cite{Hong,Mishra} are completely different from ours, and our main results are more general. In particular, our proofs rely on basic results from quantum harmonic analysis, like the Werner-Young inequality \cite{Luef,Werner}.

We denote by $\rho$ the Schr\"odinger representation of $\Rdd$, where for $(x,\omega)\in\Rdd$ and $g\in\LtRd$\[\rho(x,\xi)g(t)=e^{-\pi i x\cdot \xi}e^{2\pi it\cdot\xi}g(t-x).\] 
The Fourier-Wigner transform, also called the Fourier-Weyl transform, of a trace class operator $S$ on $\LtRd$ is defined as
\begin{equation}\label{eq:FourierWigner}
    \Fw (S)(z)=\mathrm{tr}(\rho(z)^*S),
\end{equation}
where $\mathrm{tr}$ is the standard semifinite trace on $\mathcal{B}(\LtRd)$. The most suitable Fourier transform on the phase space $\Rdd$ is the symplectic Fourier transform 
\begin{equation*}
    \Fs (G)(z)=\iint_{\Rdd}G(z^\prime)e^{-2\pi i\sigma(z,z^\prime)}dz^\prime,
\end{equation*}
where $\sigma$ denotes the standard symplectic form on $\Rdd$. Our main result is the following equivalence theorem.
\begin{theorem}\label{equi-Thm}
Suppose $\mu$ is a compactly supported Radon measure on $\R^{2d}$, and $1\leq p,q\leq \infty$. Then the following statements are equivalent:
\begin{enumerate}[i)]
    \item There exists a constant $C_\sigma>0$ such that for any $G\in L^p(\R^{2d})$,
    \begin{equation*}
        \|\F_\sigma(G)\|_{L^{q}(\mu)}\leq C_\sigma\|G\|_{L^p(\R^{2d})}.
    \end{equation*}
    \item There exists a constant $C_W>0$ such that for any $S\in \cS^p$,
    \begin{equation*}
        \|\F_W(S)\|_{L^{q}(\mu)}\leq C_W\|S\|_{\cS^p}.
    \end{equation*}
\end{enumerate}
\end{theorem}
The operation dual to Fourier restriction is that of extension and is formally given by the Fourier extension operator: $\E_\sigma(G)=\F_\sigma(G\,d\mu)$ for the symplectic Fourier transform $\Fs$. The operator analog of this extension operator is
\begin{equation}\label{QuantumExtensionOperator}
    \E_W(G)=\F_W^{-1}(G\,d\mu)=\int G(z)\rho(z)\,d\mu(z),
\end{equation}
which we understand in the weak sense and refer to as the quantum extension operator. This is the adjoint operator of the quantum restriction operator by Proposition \ref{Duality of restriction operators}. By duality, there is the following reformulation of Theorem \ref{equi-Thm} in terms of extension operators. See for instance \cite{Demeter,Schlag,Stein,Wolff} for the classical equivalence and Proposition \ref{Res-to-Ext-quantum} for the quantum case.
\begin{theorem}\label{equi-Cor}
    Let $\mu$ be a compactly supported Radon measure on $\R^{2d}$. For $1\le p,q\le\infty$ with dual exponents $p^\prime,q^\prime$, the following statements are equivalent:
\begin{enumerate}[i)]
    \item There exists a constant $C_\sigma>0$ such that for any $G\in L^p(\R^{2d})$,
    \begin{equation*}
        \|\F_\sigma(G)\|_{L^{q}(\mu)}\leq C_\sigma\|G\|_{L^p(\R^{2d})}.
    \end{equation*}
    \item There exists a constant $C_W>0$ such that for any $T\in \cS^p$,
    \begin{equation*}
        \|\F_W(T)\|_{L^{q}(\mu)}\leq C_W\|T\|_{\cS^p}.
    \end{equation*}
    \item There exists a constant $C_\sigma>0$ such that for any $G\in L^{q'}(\mu)$,
    \begin{equation*}
        \|\E_\sigma (G)\|_{L^{p'}(\R^{2d})}\leq C_\sigma\|G\|_{L^{q'}(\mu)}.
    \end{equation*}
    \item There exists a constant $C_W>0$ such that for any $G\in L^{q'}(\mu)$,
    \begin{equation*}
        \|\E_W(G)\|_{\cS^{p'}}\leq C_W\|G\|_{L^{q'}(\mu)}.
    \end{equation*}
\end{enumerate}
\end{theorem}
\noindent 
We state a consequence of the preceding Theorem \ref{equi-Cor} for operators of the form \eqref{QuantumExtensionOperator} by utilising classical Fourier restriction estimates for functions.
In particular, a restriction estimate for fractal, compactly supported, probability measures yields geometric and Fourier analytic conditions which ensure Schatten class estimates for the quantum extension operator. The following theorem is a reformulation of Theorem $3.3$ in \cite{Laba}, which extends to the setting of Schatten class operators.
\begin{theorem}\label{QuantumBakTheorem}
    Let $\mu$ be a compactly supported probability measure such that for $0<\beta\leq \alpha<2d$ and constants $C_\alpha,C_\beta>0$ we have
    \begin{align*}
        \mu(B(z,r))\leq&\, C_\alpha r^\alpha\, \text{for all }\, z\in \R^{2d}\, \text{and }\, r>0,\\
        |\F_\sigma(\mu)(\zeta)|\leq&\, C_\beta (1+|\zeta|)^{-\frac{\beta}{2}}\, \text{for all }\,\zeta\in\R^{2d}.
    \end{align*}
    Then for all
    \begin{equation*}
        p'\geq  \frac{2(4d-2\alpha+\beta)}{\beta},
    \end{equation*}
    there exists $C=C(p)>0$ such that
    \begin{equation*}
        \|\E_W(G)\|_{\cS^{p'}}\leq C\|G\|_{L^{2}(\mu)},
    \end{equation*}
    for all $G\in L^2(\mu)$.
\end{theorem}
It turns out that the extension formulation of Fourier restriction in the setting of quantum harmonic analysis has applications to different quantization schemes. For a fixed function $G$, the Weyl symbol of the operator $\E_W(G)$ is the classical extension operator $\E_\sigma(G)$. By considering the constant function $G\equiv 1$, we can deduce Schatten class estimates for operators whose Weyl symbols are the symplectic Fourier transform of compactly supported Radon measures by combining Theorem \ref{equi-Cor} with classical restriction results. We can furthermore use a quantum harmonic Tauberian theorem proved by Luef and Skrettingland in \cite{Luef} to characterize the compactness of such operators in terms of decay conditions of the Fourier transform of the measure at infinity.

\begin{theorem}\label{Char-measure-compact}
    Let $\mu$ be a compactly supported Radon measure on $\R^{2d}$. Then the operator
    \begin{equation}\label{WeylOfFourierMeasure}
        L_{\F_\sigma(\mu)}=\int \rho(z)d\mu(z)
    \end{equation}
    is compact if and only if $\F_\sigma(\mu)\in C_0(\R^{2d})$. Moreover, $\mathcal{F}_\sigma(\mu)\in L^p(\R^{2d})$ if and only if $L_{\F_\sigma(\mu)}\in \cS^p$ for $1\leq p\leq\infty$.
\end{theorem}
We have as an immediate consequence:
\begin{corollary}
    Assume there exists $\beta>0$ and a constant $C>0$ such that
    \[
    |\F_\sigma(\mu)(\zeta)|\leq C(1+|\zeta|)^{-\frac{\beta}{2}},
    \]
    for all $\zeta\in\R^{2d}$. Then $L_{\F_\sigma(\mu)}\in \cS^p$ for $p>4d/\beta$.
\end{corollary}
Mishra and Vemuri have also studied the operator \eqref{WeylOfFourierMeasure} in \cite{Mishra,Mishra2024} for compact hypersurfaces in $\R^{2d}$ with strictly positive Gaussian curvature and smooth measures $\mu$. They showed that for $d\geq 2$ the respective operator is compact, and for $p>d\geq 6$ the operator $L_{\F_\sigma(\mu)}\in\cS^{p}$. Furthermore, they conjectured that $L_{\F_\sigma(\mu)}\in \cS^{p}$ if and only if $p>4d/(2d-1)$. This corresponds to the same range as in the Fourier restriction conjecture on phase space.

By the concrete form of $L_{\F_\sigma(\mu)}$ where $\mu$ is the surface measure on $\mathbb{S}^{2d-1}$ and results about Laguerre functions, Mishra and Vemuri showed that $p>4d/(2d-1)$ is a necessary condition. A sufficient condition of the conjecture follows from Theorem \ref{Char-measure-compact} through the method of stationary phase, see Theorem $1.2.1$ in \cite{Sogge}. We therefore answer the conjecture by Mishra and Vemuri in the affirmative as a consequence of Theorem \ref{Char-measure-compact}.
\begin{theorem}\label{ConjectureTheorem}
    Suppose $S$ is a compact hypersurface with non-vanishing Gaussian curvature in $\R^{2d}$ equipped with a smooth measure $\mu$. Then the operator
    \[
    \int \rho(z)d\mu(z)\in \cS^{p},
    \]
    if and only if $p>4d/(2d-1)$.
\end{theorem}

\section{Background}\label{se:background}
\subsection{Measures}

Throughout this paper, we will consider various Borel measures on $\mathbb{R}^{2d}$. We refer to \cite{Folland-Real} for the relevant background on measure theory.

Let $\mu$ be a Borel measure on $\mathbb{R}^{2d}$. For $1\leq p<\infty$, we define the $L^p$-norm of a measurable function $F$ on $\Rdd$  with respect to $\mu$ by,
\[
\|F\|_{L^{p}(\mu)}=\left(\int_{\Rdd }|F|^pd|\mu|\right)^{1/p},
\]
and we denote by $L^p(\mu)$ the space of measurable functions $F$ such that $\|F\|_{L^p(\mu)}<\infty$. 
We extend the definition to $p=\infty$ in the natural way. We simply write  $L^p(\R^{2d})$ whenever we use the Lebesgue measure on $\R^{2d}$.

Recall that a Borel measure $\mu$ on $\Rdd$ is called finite if $|\mu|(\R^{2d})<\infty$. Furthermore, a Radon measure is a Borel measure which is finite on every compact set, inner regular on every open set and outer regular on all Borel sets. By \cite[Thm. 7.8]{Folland-Real}, we have that every Borel measure on $\R^{2d}$ that is finite on compact sets is a Radon measure. In fact, every finite measure is a Radon measure by the monotonicity property of measures.

For a finite Borel measure $\mu$ on $\Rdd$, we define the Fourier transform at a point $\zeta\in\R^{2d}$ by
\[
\F(\mu)(\zeta)=\widehat{\mu}(\zeta):=\int e^{-2\pi i \zeta\cdot z}d\mu(z),
\]
which is a continuous function on $\R^{2d}$. If $\mu$ is absolutely continuous with respect to the Lebesgue measure, then there exists by the Radon-Nikodym Theorem an  $F\in L^1(\R^{2d})$ such that $d\mu=Fdx$. In this case, we identify $\widehat{\mu}=\widehat{F}$ with the regular Fourier transform of the $L^1$-function $F$, and thus $\widehat{\mu}\in C_0(\R^{2d})$ by the Riemann-Lebesgue Lemma. However, this is not the case for an arbitrary finite Borel measure, as can be seen by the fact that $\widehat{\delta_0}=1$. Here $\delta_0$ denotes the point measure at $z=0$.

Finally, a finite Borel measure is called a Rajchman measure if $\widehat{\mu}\in C_0(\R^{2d})$. Thus, every absolutely continuous measure with respect to the Lebesgue measure is a Rajchman measure.

We denote the Schwartz class of rapidly decaying smooth functions on $\Rd$ by $\mathscr{S}(\R^d)$, and its dual space of tempered distributions by $\mathscr{S}^\prime(\R^{d})$. We use a sesquilinear dual-paring for $\mathscr{S}(\R^d)$ and $\mathscr{S}^\prime(\R^d)$ which is compatible with the $L^2(\Rd)$ inner product. That is, given $\psi\in\mathscr{S}(\R^d)$, we can identify $\psi$ with a tempered distribution $\tau_\psi\in\mathscr{S}'(\R^d)$ by
\[
\tau_\psi(\varphi)=\int_{\R^d} \varphi(x)\overline{\psi(x)}dx=\langle \varphi,\psi\rangle,\qquad\text{for all}~~\varphi\in\mathscr{S}(\R^d).
\]
\subsection{The Symplectic Fourier Transform}
The standard symplectic form on $\R^{2d}$ is given by
\[
\sigma\left((x,\xi),(y,\eta)\right)=y\xi-x\eta=J(x,\xi)\cdot (y,\eta),
\]
where $J\in SO(2d)$ denotes the symplectic matrix
\begin{equation*}
    J=\begin{pmatrix}0&I_d\\-I_d&0\end{pmatrix},
\end{equation*}
and where $I_d$ is the identity matrix on $\R^d$. For a finite Borel measure $\mu$ on $\R^{2d}$, we define the symplectic Fourier transform at a point $\zeta\in \R^{2d}$ as
\begin{equation*}
    \F_\sigma(\mu)(\zeta)=\int e^{-2\pi i \sigma(\zeta,z)}d\mu(z)=\F(\mu)(J\zeta),
\end{equation*}
where $\mathcal{F}$ denotes the standard Fourier transform.

As the symplectic Fourier transform only differs from the standard Fourier transform by a rotation of the variables, most classical results are still true for the symplectic case. However, from the anti-symmetry of the symplectic form, there is a sign change in Parseval's identity and the Fourier inversion formula. In particular, we have $\F_{\sigma}^{-1}=\F_{\sigma}$, and there is the following version of Parseval's identity for finite measures.
\begin{lemma}[Lemma $3.5$ in \cite{Wolff}]\label{Plancherelfor finite measures}
    Let $\mu$ and $\nu$ be finite Borel measures on $\R^{2d}$. Then
\begin{equation}
        \int \F_{\sigma}(\mu)(z)d\nu(z)=\int \F_{\sigma}(\nu)(-\zeta)d\mu(\zeta).
    \end{equation}
\end{lemma}

Unlike for the standard Fourier transform, there is a sign change in Parseval's identity for the symplectic Fourier transform. By instead considering a sesquilinear dual-pairing we recover the standard Parseval's identity on $L^2(\Rdd)$.
\begin{corollary}
    If $\Phi,\Psi\in L^2(\R^{2d})$, then
    \begin{equation*}
        \iint_{\R^{2d}} \F_\sigma(\Phi)(x)\overline{\Psi(x)}dx=\iint_{\R^{2d}} \Phi(\xi)\overline{\F_\sigma(\Psi)(\xi)}d\xi.
    \end{equation*}
\end{corollary}
This indicates that a sesquilinear dual-paring might be more natural whenever we are working with symplectic structures. We have therefore chosen to only consider sesquilinear dual-parings throughout the rest of the paper.

Since the Lebesgue measure is invariant under actions of $\mathrm{SO}(2d)$, it is possible to extend all classical Fourier restriction estimates to the symplectic case. This is the content of the next lemma, and is a consequence of the fact that the symplectic Fourier transform is the standard Fourier transform up to a rotation.
\begin{lemma}\label{ExtensionOfResToSymp} Let $\mu$ be a Borel measure on $\R^{2d}$. If there exists $1\leq p,q\leq \infty$ and a constant $C>0$ such that
\[\|\F(\Psi)\|_{L^q(\mu)}\leq C\|\Psi\|_{L^p(\R^{2d})},\]
for all $\Psi\in \mathscr{S}(\R^{2d})$,
then the same estimate holds for the symplectic Fourier transform, meaning
\[\|\F_\sigma(\Psi)\|_{L^q(\mu)}\leq C\|\Psi\|_{L^p(\R^{2d})},\]
for all $\Psi\in\mathscr{S}(\R^{2d})$, and the same $p,q$ and $C$.
\end{lemma}

\subsection{The Heisenberg group, time-frequency analysis and quantization}

Let us briefly review some basic facts on the representation theory of the Heisenberg group $\mathbb{H}$, see for example  \cite{Folland,Thangavelu_book}, where $\mathbb{H}$ is the group $\Rd\times\Rd\times\R$ with multiplication
\[ (x,\xi,s)\bullet(x^\prime,\xi^\prime,s^\prime)=\left(x+x^\prime,\xi+\xi^\prime,s+s^\prime+\frac{x^\prime\cdot\xi-x\cdot\xi^\prime}{2}\right).\]
We denote by $\rho_\lambda$ the Schr\"odinger representation of the Heisenberg group $\mathbb{H}$ on $\LtRd$, defined by 
\[\rho_\lambda(z,s)g(t)=e^{2\pi i\lambda s }e^{\lambda(-\pi ix\cdot\xi+2\pi it\cdot\xi)}g(t-x)\qquad\text{for}~~z=(x,\xi)\in\Rdd.\] 
Note that the restriction of $\rho_\lambda$ to $\Rdd$ defines a projective unitary representation of $\Rdd$ on $\LtRd$, denoted by $\rho_\lambda(z):=\rho_\lambda(z,0)$. The group Fourier transform on $\mathbb{H}$ is given by
\begin{equation*}
    \widehat{F}(\lambda)g(t)=\iint_{\mathbb{H}}F(z,s)\rho_\lambda(z,s)g(t)\,dz\,ds,\qquad\text{for}~~F\in L^1(\mathbb{H}),\,g\in\LtRd.
\end{equation*}
For $F\in L^1(\mathbb{\mathbb{H}})\cap L^2(\mathbb{\mathbb{H}})$ the mapping $\widehat{F}(\lambda)\mapsto F(z,s)$ computes the Fourier coefficient of $F$, which for a fixed $s$ coincides with the Fourier-Wigner transform of $F$, see Section \ref{sec:QHA}.

We include a brief review of notation and results from time-frequency analysis. The theory presented can be found in \cite{Grochenig}. From now on we fix $\lambda=1$.

For $z=(x,\xi)\in\R^{2d}$, we denote the Schr\"{o}dinger representation by 
\begin{equation}\label{SchrodingerRepresentation}
    \rho(z)=\rho(x,\xi)=e^{-\pi i x\cdot \xi}M_{\xi}T_x,
\end{equation}
where $M_\xi$ and $T_x$ are the modulation and translation operators on $L^2(\mathbb{R}^d)$, respectively. 

The Schr\"{o}dinger representation gives rise to a time-frequency representation. For $f,g\in L^2(\R^d)$, the \textit{cross-ambiguity function} of $f$ and $g$ is defined as
\begin{equation}\label{AmbiguityFunc}
    A(f,g)(x,\xi)=\langle f,\rho(x,\xi) g\rangle=\int_{\R^d}f\left(t+\frac{x}{2}\right)\overline{g\left(t-\frac{x}{2}\right)}e^{-2\pi i \xi\cdot t}dt.
\end{equation}
For fixed $g\in\mathscr{S}(\R^d)$, the cross-ambiguity function extends to a bounded linear functional on $\mathscr{S}'(\R^{d})$. The next result is known as the covariance property of the cross-ambiguity function.
\begin{lemma}[Lemma $3.1.3$ in \cite{Grochenig}]\label{Covariance property}
    Let $f,g\in L^2(\R^d)$. Then for any $z,\zeta\in \R^{2d}$
    \begin{equation*}
        A(\rho(\zeta)f,g)(z)=e^{\pi i \sigma(\zeta,z)}A(f,g)(z-\zeta).
    \end{equation*}
\end{lemma}

An important case is when both $f$ and $g$ are the $L^2$ normalized Gaussian,
\begin{equation}\label{Gaussian}
    g_0(t)=2^{\frac{d}{4}}e^{-\pi|t|^2}.
\end{equation}
Then the cross-ambiguity function is again a Gaussian,
\begin{equation}\label{STFTgaussian}
A(g_0,g_0)(z)=e^{-\pi\frac{|z|^2}{2}}.
\end{equation}
This is a special case of Lemma $1.5.2$ in \cite{Grochenig}. 

A key property of the cross-ambiguity function is Moyal's identity,
\[
\langle A(f,\varphi),A(g,\psi )\rangle=\langle f,g\rangle\langle \psi,\varphi\rangle,
\]
for all $f,g,\varphi,\psi\in L^2(\R^d)$. This result in an inversion formula, see for example \cite[Cor. 3.2.3]{Grochenig}.
\begin{lemma}[Inversion formula]\label{inversion formula}
Let $g,h\in L^2(\R^d)$ with $\langle g,h\rangle\neq 0$. Then for all $f\in L^2(\R^d)$
\begin{equation*}
    f=\frac{1}{\langle h,g\rangle}\iint_{\R^{2d}}A(f,g)(z) \rho(z)h \,dz,
\end{equation*}
where the integral is defined weakly.
\end{lemma}

A localization operator with symbol $a\in \mathscr{S'}(\R^{2d})$ and windows $\varphi,\psi\in \mathscr{S}(\R^d)$ is the operator $\mathcal{A}_a^{\varphi,\psi}:\mathscr{S}(\R^d)\to\mathscr{S}'(\R^d)$ such that
\begin{equation*}
    \langle \mathcal{A}_a^{\varphi,\psi}f,g\rangle=\langle a A(f,\varphi),A(g,\psi)\rangle,
\end{equation*}
holds for all $f,g\in \mathscr{S}(\R^d)$. When the symbol $a$ is a function on $\R^{2d}$, the localization operator can be written as
\begin{equation}\label{LocalizationOperator}    \mathcal{A}_a^{\varphi,\psi}f=\iint_{\R^{2d}}a(z)A(f,\varphi)(z)\rho(z)\psi\,dz,
\end{equation}
when acting on $f\in L^2(\R^d)$. The integral has to be understood in the weak sense. If $a\equiv 1$, and $\varphi=\psi$ are chosen such that $\|\varphi\|_{L^2}=1$, then the localization operator becomes the identity operator on $L^2(\R^d)$ by Lemma \ref{inversion formula}. The mapping $a\mapsto \mathcal{A}_a^{\varphi,\psi}$ is often referred to as Berezin quantization and in the case $\varphi=\psi=g_0$ as the anti-Wick quantization.

Associated to the cross-ambiguity function is another important time-frequency representation, known as the cross-Wigner distribution. Given $f,g\in\mathscr{S}(\R^d)$ and $(x,\xi)\in\R^{2d}$, the cross-Wigner distribution is defined as
\begin{equation*}
    W(f,g)(x,\xi)=\int_{\R^d}f\left(x+\frac{t}{2}\right)\overline{g\left(x-\frac{t}{2}\right)}e^{-2\pi i \xi\cdot t}dt.
\end{equation*}
It was first introduced by Wigner in 1932 in an attempt to achieve a phase space representation of quantum mechanics \cite{Wigner}. The cross-Wigner distribution is related to the cross-ambiguity function through the symplectic Fourier transform.
\begin{lemma}[Lemma $4.3.4$ in \cite{Grochenig}]\label{FourierOfWigner}
For any $f,g\in\mathscr{S}(\R^d)$ and all $(x,\xi)\in\R^{2d}$,
\begin{equation*}
    W(f,g)(z)=\F_\sigma(A(f,g))(z)=\F_\sigma(\langle f,\rho(z)g\rangle).
\end{equation*}
\end{lemma}
An important example is the Wigner distribution of the $L^2$ normalized Gaussian $g_0$. Utilizing equation \eqref{STFTgaussian} it follows that
\begin{equation}\label{WignerGaussian}
    W(g_0,g_0)(z)=\mathcal{F}_\sigma(A(g_0,g_0))(z)=\mathcal{F}_\sigma\left(e^{-\pi\frac{|\cdot|^2}{2}}\right)(z)=2^{d}e^{-2\pi|z|^2}.
\end{equation}

Just as the cross-ambiguity function gave rise to the localization operator, the cross-Wigner distribution gives rise to another quantization scheme. For a tempered distribution $a\in\mathscr{S}'(\R^{2d})$ the Weyl quantization is defined as the operator $L_a:\mathscr{S}(\R^d)\to\mathscr{S}'(\R^{2d})$ given by  
\[L_{a}=\iint_{\R^{2d}}\mathcal{F}_\sigma(a)(z)\rho(z)dz,\]
where $\rho$ is the Schr\"{o}dinger representation as defined in \eqref{SchrodingerRepresentation}.
We refer to $a$ as the Weyl symbol of $L_a$. A well-known result by Pool states that the Weyl quantization extends to an isometric isomorphism between $L^2(\R^{2d})$ and the Hilbert-Schmidt operators on $L^2(\R^d)$, see \cite{Pool}. 

A different representation of the Weyl-quantization is the following \cite[Prop. 2.5]{Folland}:
\begin{prop}[Weak formulation]\label{FollandProp}
Let $a\in \mathscr{S}'(\R^{2d})$, and $f,g\in\mathscr{S}(\R^d)$. Then
\begin{equation*}
    \langle L_a f,g\rangle=\langle a, W(g,f)\rangle,
\end{equation*}
where we use a sesquilinear dual pairing between $\mathscr{S}(\Rd)$ and $\mathscr{S}^\prime(\Rd)$.
\end{prop}
Let us close this section with a remark related to the Weyl symbol of a  localization operator:
\begin{equation}\label{eq:LocOpWeyl}
\mathcal{A}^{\varphi,\psi}_a=L_{a*W(\psi,\varphi)},    
\end{equation}
i.e. the Weyl symbol of the localization operator is $a*W(\psi,\varphi)$, see  \cite[Thm. 5.2]{Luef_mixed}.
\subsection{Schatten Class Operators}

We denote the bounded operators on $L^2(\R^d)$ by $\mathcal{B}(L^2(\R^d))$ and the compact operators by $\mathcal{K}$. Recall that any compact operator $T$ on $L^2(\R^d)$ has a singular value decomposition (see for instance \cite[Chap.3]{SimonOp}):
\[
T=\sum_{n\in\mathbb{N}}s_n(T)\varphi_n\otimes \psi_n,
\]
where $\{s_n(T)\}_{n\in\mathbb{N}}$ denotes the sequence of the singular values of $T$ and $\{\varphi_n\}_{n\in\mathbb{N}},\{\psi_n\}_{n\in\mathbb{N}}$ are two orthonormal bases for $L^2(\R^d)$. For a general compact operator $T$ it is well-known that
\[
\lim_{n\to\infty} s_n(T)=0.
\]
By requiring summability conditions on the singular values of the operators $T\in\mathcal{K}$ we can define the Schatten $p$-classes. For any $1\leq p<\infty$ we define
\[
\cS^p:=\left\{T\in \mathcal{K}: \{s_n(T)\}_{n\in\mathbb{N}}\in \ell^p\right\},
\]
and for $p=\infty$ we use the identification of $\cS^\infty$ with the space of bounded operators $\mathcal{B}(L^2(\R^d))$. There is the inclusion $\cS^1\subseteq \cS^p\subseteq \cS^q$ when $1\leq p\leq q\leq \infty$, which comes from the inclusion relations for $\ell^p$ spaces.

For $T\in \cS^1$ one can define a trace
\[
\tr(T)=\sum_{n\in\mathbb{N}}\langle Te_n,e_n\rangle,
\]
for some orthonormal basis $\{e_n\}_{n\in\mathbb{N}}$ on $L^2(\R^d)$. In fact, the trace is independent of the choice of the orthonormal basis. Through the trace one can define the norm
\[
\|T\|_{\cS^p}=\tr(|T|^p)^\frac{1}{p}=\tr\left((T^*T)^{\frac{p}{2}}\right)^\frac{1}{p}=\left(\sum_{n\in\mathbb{N}}s_n(T)^p\right)^\frac{1}{p},
\] 
which turns $\cS^p$ into a Banach space. It is worth noting that finite rank operators are dense in $\cS^p$ for $1\leq p<\infty$, as the space $c_{00}$ of sequences with only finitely many non-zero elements are dense in $\ell^p$ for the same range of $p$. This implies that $\cS^1$ is also dense in $\cS^{p}$ for $1< p<\infty$. 

There is an analog of H\"{o}lder's inequality for Schatten $p$-classes.
\begin{theorem}[Theorem $3.7.6$ in \cite{SimonOp}]
    Let $1<p,q,r<\infty$ be such that $\frac{1}{p}+\frac{1}{q}=\frac{1}{r}$.
    If $S\in \cS^p$ and $T\in \cS^q$, then $ST\in \cS^r$ and
    \begin{equation}
        \|ST\|_{\cS^r}\leq\|S\|_{\cS^p}\|T\|_{\cS^q}.
    \end{equation}
\end{theorem}
As in the case of $\ell^p$-sequence spaces, this allows one to define a dual pairing between $\mathcal{S}^p$ and $\mathcal{S}^{p'}$ for the conjugate exponent $p'=p/(p-1)$ with the same modifications at the endpoints $p=1$ and $p=\infty$:

\begin{theorem}[Theorem $3.6.8$ and $3.7.4$ in \cite{SimonOp}]
There is the following characterization of the dual space of $\cS^p$ for $1\leq p \leq \infty$.
\begin{enumerate}[a)]
    \item We have $\mathcal{K}^\prime=\cS^1$ and $(\cS^1)'=\mathcal{B}(L^2(\R^d))$ where the dual pairing is given by
\[
\langle S,T\rangle_{\cS^1,\mathcal{K}} =\mathrm{tr}(ST^*).
\]
    \item Let $1<p<\infty$ and $p'=p/(p-1)$. Then for any $S\in \cS^p$ and $T\in \cS^{p'}$ we have $ST^*\in \cS^1$ and
\begin{equation}\label{HolderSchatten}
    \left|\langle S,T\rangle_{\cS^p,\cS^{p'}}\right|=\left|\mathrm{tr}(ST^*)\right|\leq \|ST^*\|_{\cS^1}\leq \|S\|_{\cS^p}\|T\|_{\cS^{p'}}.
\end{equation}
Moreover, for any $S\in\cS^p$,
\begin{equation}\label{SchattenDualNormEst}
    \|S\|_{\cS^p}=\sup_{\|T\|_{\cS^{p'}}=1}\left|\mathrm{tr}(ST^*)\right|.
\end{equation}
\end{enumerate}
\end{theorem}
This implies that, for $1<p<\infty$, the dual space of $\cS^p$ is $\cS^{p'}$ where the dual pairing is given by
\[
\langle S,T\rangle_{\mathcal{S}^p,\cS^{p'}}= \tr(ST^*),\qquad\text{for}~~S\in\cS^p,\,T\in\cS^{p'},
\]
and note that $\cS^p$ is reflexive for $1<p<\infty$. Furthermore, $\cS^2$ is a Hilbert space under this pairing and we denote this inner product by $\langle\cdot,\cdot\rangle_{\mathcal{HS}}$ and refer to $\cS^2$ as the space of Hilbert-Schmidt operators.

\subsection{Basics on Quantum Harmonic Analysis}\label{sec:QHA}

Let us review the key notions and basic results of Quantum Harmonic Analysis (QHA), which provides a convenient setting for harmonic analysis of operators \cite{Werner}. The starting point is the convolution of operators. More concretely, operator-operator convolution and operator-function convolution are defined for $S,T\in\cS^1$ and $F\in L^1(\R^{2d})$ by
\[
S\star T(z)=\tr(S\rho(z)PTP\rho(-z)),\quad F\star S=\iint_{\mathbb{R}^{2d}} f(z)\rho(z)S\rho(-z)dz\qquad\text{for}~~z\in\R^{2d},
\]
where $Pf(x)=f(-x)$ denotes the parity operator, and the integral is understood weakly. It follows from the definition that these convolutions are commutative. They are also associative, see for instance \cite{Luef,Werner}.
\begin{prop}[Proposition $2.1$ in \cite{Luef}]
    For $S,T,R\in\cS^1$ and $F,G\in L^1(\R^{2d})$,
    \begin{align*}
        (R\star S)\star T=&R\star(S\star T),\\
        F*(R\star S)=&(F\star R)\star T,\\
        (F*G)\star R=&F\star(G\star R).
    \end{align*}
\end{prop}
An important tool in classical harmonic analysis is Young's inequality for $L^p$-functions. We state here an analog for the setting of quantum harmonic analysis, see \cite[Prop. 3.2(v)]{Werner} or \cite[Prop. 2.2]{Luef}.
\begin{prop}[Werner-Young's inequality]
Suppose $S\in \cS^p$, $T\in \cS^q$ and $F\in L^p(\R^{2d})$ with $1+\tfrac{1}{r}=\tfrac{1}{p}+\tfrac{1}{q}$.

Then $S\star T\in L^r(\R^{2d})$, $F\star T\in \cS^r$, and
\begin{align}
    &\|S\star T\|_{L^{r}}\leq \|S\|_{\cS^p}\|T\|_{\cS^q}\label{Young op-op},\\
    &\|F\star T\|_{\cS^r}\leq \|F\|_{L^p}\|T\|_{\cS^q}\label{Young op-func}.
\end{align}
\end{prop}

For $T\in\cS^1$ we define the Fourier-Wigner transform of $T$ at a point $z\in\R^{2d}$ by
\[
\F_W(T)(z)=\tr(T\rho(-z))=\langle T,\rho(z)\rangle_{\mathcal{\mathcal{HS}}}.
\]
This defines a bounded function on $\R^{2d}$ as
\[
|\F_W(T)(z)|=|\tr(T\rho(-z))|\leq \|T\|_{\cS^1}\|\rho\|=\|T\|_{\cS^1},
\]
by \cite[Thm. 3.6.6]{SimonOp}. Moreover, for $T\in \cS^1$ it follows that $\F_W(T)\in C_0(\R^{2d})$ by an analogue of the Riemann-Lebesgue lemma found in \cite{Werner}. By a known result of Pool\footnote{Pool showed that the Weyl quantization $a\mapsto L_a$ extends to a unitary operator from $L^2(\R^{2d})$ to $\cS^2$. For $a\in\mathscr{S}(\Rdd)$, the Weyl quantization may be written as $L_a=\F_W^{-1}(\F_\sigma(a))$, and thus the Fourier-Wigner transform defines a unitary operator from $\cS^2$ to $L^2(\Rdd)$ by the isomorphism of the symplectic Fourier transform on $L^2(\Rdd)$.} \cite{Pool}, the Fourier-Wigner transform extends to a unitary operator from $\cS^2$ to $L^{2}(\R^{2d})$. We therefore obtain Hausdorff-Young's inequality through interpolation for operator convolutions, \cite[Prop. 3.4]{Werner}.
\begin{prop}[Werner--Hausdorff--Young Theorem]\label{Operator Hausdorff-Young}
    Let $1\leq p\leq 2$, and $p^{-1}+p'^{-1}=1$. If $T\in \cS^p$, then $\F_W(T)\in L^{p'}(\R^{2d})$ and
    \begin{equation}\label{Hausdorff-Young Operator}
        \|F_W(T)\|_{L^{p'}(\R^{2d})}\leq \|T\|_{\cS^p}.
    \end{equation}
\end{prop}
Let us look at the most elementary instance of a Fourier-Wigner transform.
\begin{example}
Let $f,g\in L^2(\mathbb{R}^d)$ and consider the rank one operator $f\otimes g\in \cS^1$. Then we have that
\begin{equation}\label{Fourier-Wigner-rank-one}
    \F_W(f\otimes g)=A(f,g),
\end{equation}
where $A(f,g)$ denotes the cross-ambiguity function of $f$ and $g$ as defined in Equation \eqref{AmbiguityFunc}.
\end{example}

The Fourier-Wigner transform interacts with convolutions in the same fashion as the classical Fourier transform. The following intertwining properties of the Fourier-Wigner transform and convolutions are proven in \cite[Prop. 3.4]{Werner}.
\begin{prop}[Werner's convolution theorem]\label{IntertwiningProp}
    Let $S,T\in\cS^1$ and $F\in L^1(\R^{2d})$. Then
    \begin{align}
        \F_\sigma(S\star T)=&\F_W(S)\F_W(T)\\
        \F_W(F\star S)=&\F_\sigma(F)\F_W(S).
    \end{align}
\end{prop}

\subsection{Wiener's Tauberian Theorem}
Wiener's Tauberian theorem is a classical result in harmonic analysis \cite{Wiener}. The theorem states that if $f\in L^\infty(\R^d)$, and there exists $h\in L^1(\R^d)$ with non-vanishing Fourier transform and $A\in\mathbb{C}$ such that 
\[
\lim_{x\to \infty} (f*h)(x)=A\int_{\R^d} h(x)dx,
\]
then for any $g\in L^1(\R^d)$,
\[
\lim_{x\to \infty} (f*g)(x)=A\int_{\R^d} g(x)dx.
\]
The proof is based on Wiener's approximation theorem. For $f\in L^1(\R^d)$ it states that $\overline{\text{span}}\{T_xf:x\in\R^d\}=L^1(\R^d)$ if and only if $\F(f)(\xi)\neq 0$ for any $\xi\in\R^d$.

In \cite{Luef}, Luef and Skrettingland extended Wiener's Tauberian theorem to the setting of operators, where they proved the following.
\begin{theorem}[Theorem $4.1$ in \cite{Luef}]\label{TauberianTheorem}
    Let $F\in L^\infty(\R^{2d})$, and assume that one of the following equivalent statements holds for some $A\in\mathbb{C}$.
    \begin{enumerate}[i)]
        \item There is some $S\in\cS^1$, with $\F_W(S)(z)\neq 0$ for all $z\in \R^{2d}$, such that
        \begin{equation*}
            F\star S=A\cdot\mathrm{tr}(S) I_{L^2}+K
        \end{equation*}
        for some $K\in\mathcal{K}$.
        \item There is some $a\in L^1(\R^{2d})$, with $\F_\sigma(a)\neq 0$ for all $z\in \R^{2d}$, such that
        \begin{equation*}
            F*a=A\iint_{\R^{2d}} a(z)dz+H
        \end{equation*}
        for some $H\in C_0(\R^{2d})$.
    \end{enumerate}
    Then both of the following statements hold:
    \begin{enumerate}[a)]
        \item For any $T\in\cS^1$, $F\star T=A\cdot\mathrm{tr}(T)I_{L^2}+K_T$ for some $K_T\in\mathcal{K}$.
        \item For any $G\in L^1(\Rdd)$, $F*G=A\iint_{\R^{2d}} G(z)dz+H_G$ for some $H_G\in C_0(\R^{2d})$.
    \end{enumerate}
\end{theorem}
Here $I_{L^2}$ denotes the identity operator on $L^2(\R^{d})$. In a recent preprint \cite{Fulsche_Luef_Werner_24} this result was extended using Werner's correspondence theory and the theory of limit operators. 
\section{Fourier restriction for Schatten class operators}

In this section we will prove the theorems announced in Section \ref{sec:Intro} and we start with the basic insight that the Fourier restriction theorem for the symplectic Fourier transform is in fact equivalent to the Fourier restriction theorem for Schatten class operators, Theorem \ref{equi-Thm}.

\begin{proof}[Proof of Theorem \ref{equi-Thm}]
We start with an elementary observation: Since $\mu$ is compactly supported we denote $R=\text{diam}(\mathrm{supp}(\mu))<\infty$. Fix $z_0\in \R^{2d}$ such that $\text{supp}(\mu)\subset B(z_0,R)$.
\\~~\\
${ i)\Rightarrow ii):}$ We assume that there exists a constant $C_\sigma>0$ such that 
\begin{equation*}\label{First-assertion}
\|\mathcal{F}_\sigma(G)\|_{L^q(\mu)}\leq C_\sigma\|G\|_{L^p},
\end{equation*}
for every $G\in L^p(\R^{2d})$. 
Consider $T\in \cS^p$ and associate its cross-Cohen's class, which was introduced in \cite{Luef_mixed},
\[
Q_T(\varphi,\psi)(z)=(\psi\otimes \varphi)\star T(z),\qquad\text{for}~~\varphi,\psi\in L^2(\R^d).
\]
It follows from Werner-Young's inequality,  \eqref{Young op-op}, that $Q_T(\varphi,\psi)\in L^p(\R^{2d})$. Moreover, Proposition \ref{IntertwiningProp} and \eqref{Fourier-Wigner-rank-one} gives the point-wise result
\[
|\mathcal{F}_\sigma(Q_T(\varphi,\psi))(z)|=|\mathcal{F}_W(T)(z)\mathcal{F}_W(\psi\otimes \varphi)(z)|=|\F_W(T)(z)|\, |A(\psi,\varphi)(z)|.
\]
We will now consider a specific choice of $\varphi$ and $\psi$. Let $\varphi(t)=g_0(t)=2^{d/4}\exp{(-\pi |t|^2)}$ be the $L^2$-normalized Gaussian, and $\psi=\rho(z_0)g_0$ a time-frequency shifted Gaussian. In this case, it follows from \eqref{STFTgaussian} that
\[
|A(g_{z_0},g_{0})(z)|=e^{-\pi\frac{|z-z_0|^2}{2}}\geq e^{-\pi\frac{R^2}{2}},
\]
for all $z\in\text{supp}(\mu)$ by Lemma \ref{Covariance property}. We therefore have a lower bound on the point-wise estimate of $\F_\sigma(Q_T(g_0,g_{z_0}))$,
\[
|\mathcal{F}_\sigma(Q_T(g_0,g_{z_0}))(z)|=|\F_W(T)(z)|\, |A(g_{z_0},g_0)(z)|\geq e^{-\pi\frac{R^2}{2}}|\F_W(T)(z)|,
\]
for all $z\in \text{supp}(\mu)$.
Thus, by the assumption that $\mathrm{i)}$ holds, there exists a constant $C_\sigma>0$ such that
\[
e^{-q\pi\frac{R^2}{2}}\int |\F_W(T)|^qd|\mu|\leq\int |\F_\sigma(Q_T(g_{0},g_{z_0}))|^qd|\mu|\leq C_\sigma^q\|Q_T(g_{0},g_{z_0})\|_{L^p}^q\leq C_\sigma^q\|T\|_{\cS^p}^q,
\]
where we used $\|g_{z_0}\otimes g_{0}\|_{\cS^1}=\|g_{z_0}\|_{L^2}\|g_{0}\|_{L^2}=1$. This shows that $\mathrm{i)}$ implies $\mathrm{ii)}$.
\\~~\\
On the other hand, if we assume $\mathrm{ii)}$, then there exists $C_W>0$ such that
\[
\|\F_W(T)\|_{L^q(\mu)}\leq C_W\|T\|_{\cS^p},
\]
for all $T\in\cS^p$. In particular, for any $F\in L^p(\R^{2d})$ we can consider the localization operator 
\[
\mathcal{A}_F^{\varphi,\psi}=(\psi\otimes\varphi)\star F, 
\]
which is an element of $\cS^p$ by Young's inequality, \eqref{Young op-func}. By the convolution property of the Fourier-Wigner transform, Proposition \ref{IntertwiningProp},
\[
\F_W(\mathcal{A}_F^{\varphi,\psi})(z)=\F_W(\psi\otimes \varphi)(z)\F_\sigma(F)(z).
\]
By choosing $\varphi=g_0$ and $\psi=g_{z_0}$, as above, it follows from the same estimate of the ambiguity function that
\[
e^{-q\pi\frac{R^2}{2}}\int |\F_\sigma(F)|^q\,d|\mu|\leq\int |\F_W(\mathcal{A}_F^{\varphi,\psi})|^q\,d|\mu|\leq C_W^q\|\mathcal{A}_F^{\varphi,\psi}\|_{\cS^p}^q\leq C_W^q\|F\|_{L^{p}(\R^{2d})}^q,
\]
which proves $\mathrm{i)}$.
\end{proof}

A well-studied object in the literature on Fourier restriction is the so-called extension operator
\[
\E_\sigma\varphi(z)=\F^{-1}_\sigma(\varphi d\mu)(z)=\F_\sigma(\varphi d\mu)(z)=\int e^{-2\pi i \sigma(z,\zeta)}\varphi(\zeta)d\mu(\zeta),
\]
for some suitably nice functions $\varphi$, e.g. Schwartz functions. The extension operator is, at least formally, the adjoint of the restriction operator (see for instance \cite[p. 353]{Stein} or \cite[Chap. 1.1.]{Demeter}).
Note that $\E_\sigma$ extends a function on the support of $\mu$ to a function on the whole phase space $\R^{2d}$. 

A natural problem we would like to address is the following: What is the correct extension operator when considering Fourier restriction of operators?

As the Fourier-Wigner restriction operator maps a Schatten class operator to a function on the support of $\mu$, it is reasonable to expect the Fourier-Wigner extension operator to map a function on the support of $\mu$ to an operator on $L^2(\R^{d})$. This is in fact the case.

\begin{prop}\label{Duality of restriction operators}
Let $\mu$ be a Borel measure on $\R^{2d}$. Then for every $\Phi\in L^1(\mu)$ there exists $\E_W\Phi\in\mathcal{B}(L^2(\R^d))$ 
such that for any $T\in\mathcal{S}^1$,
\begin{equation}\label{Wigner-Formal-Adjoint}
    \langle \Phi,\F_W(T)\rangle_{L^2(\mu)}=\int \Phi(z)\overline{\mathcal{F}_W(T)(z)} d\mu(z)=\mathrm{tr}((\E_W\Phi) T^*)=\langle \E_W\Phi,T\rangle_{\mathcal{HS}}.
\end{equation}
The operator $\E_W\Phi$ is given by
\begin{equation}\label{DefExtension}
  \E_W\Phi=\int\Phi(z)\rho(z)d\mu(z),
\end{equation}
where the integral is defined weakly, and satisfies the norm bound $\|\E_W\Phi\|\leq \|\Phi\|_{L^1(\mu)}$.
\end{prop}
\begin{proof}
For $T\in\mathcal{S}^1$ it follows from \cite[Prop. 3.4]{Werner} that $\mathcal{F}_W(T)\in C_0(\R^{2d})$. Thus, the integral on the left hand side in \eqref{Wigner-Formal-Adjoint} is well-defined and converges for any $\Phi\in L^1(\mu)$. A formal calculation gives
\begin{align*}
    \int \Phi(z)\overline{\mathcal{F}_W(T)(z)} d\mu(z)=\int \Phi(z)\overline{\mathrm{tr}(T\rho(-z))}d\mu(z)
    =\mathrm{tr}\left(\left(\int \Phi(z)\rho(z)d\mu(z)\right)T^*\right),
\end{align*}
where we used $\rho(-z)=\rho(z)^*$.

To estimate the operator norm, we note that for any $f,g\in L^2(\R^d)$,
\begin{equation*}
|\langle\E_W\Phi(f),g\rangle|=\left|\int  \Phi(z)\langle \rho(z)f,g\rangle d\mu(z)\right|\leq \int  |\Phi(z)|\left|\langle \rho(z)f,g\rangle\right| d|\mu|(z)\leq \|\Phi\|_{L^1(\mu)}\|f\|_{L^2}\|g\|_{L^2},
\end{equation*}
as $\rho$ is a unitary operator on $L^2(\Rd)$ for any choice of $z\in \R^{2d}$. Varying over all normalized $f,g$ gives
\begin{equation*}
\|\E_W\Phi\|=\sup_{\|f\|_2=\|g\|_2=1}|\langle\E_W\Phi(f),g\rangle|\leq \|\Phi\|_{L^1(\mu)}.
\end{equation*}
\end{proof}

\begin{prop}\label{Res-to-Ext-quantum}
Let $\mu$ be a Borel measure on $\R^{2d}$. Then for any $1\leq p , q \leq \infty$ and constant $C>0$, the following are equivalent:
\begin{enumerate}[i)]
    \item For any $T\in\cS^p$ we have $\|\F_W(T)\|_{L^{q}(\mu)}\leq C\|T\|_{\cS^p}$.
    \item For any $G\in L^{q'}(\mu)$ we have $\|\E_W(G)\|_{\cS^{p'}}\leq C\|G\|_{L^{q'}(\mu)}$.
\end{enumerate}
\end{prop}
\begin{proof}
Since $\cS^1$ is dense in $\cS^p$, and simple functions are dense in $L^{q'}(\mu)$, we may  assume that $T\in\cS^1$ and $g\in L^1(\mu)\cap L^{q'}(\mu)$ to ensure that everything is well-behaved.
\\~~\\
{ i) $\Rightarrow$ ii):}
By \eqref{Wigner-Formal-Adjoint} and H\"{o}lder's inequality we have,
\begin{equation*}
    |\mathrm{tr}((\mathcal{E}_WG)T^*)|=\left|\int G\overline{\F_W(T)}d\mu\right|\leq \|\F_W(T)\|_{L^q(\mu)}\|G\|_{L^{q'}(\mu)}\leq C\|T\|_{\cS^p}\|G\|_{L^{q'}(\mu)}.
\end{equation*}
Since the dual space of $\cS^p$ is $\cS^{p'}$, it follows that
\begin{equation*}
    \|\E_W(G)\|_{\cS^{p'}}=\sup_{\|T\|_{\cS^p}=1}|\mathrm{tr}((\mathcal{E}_W(G))T^*)|\leq C\|G\|_{L^{q'}(\mu)}.
\end{equation*}
\\~~\\
{ ii) $\Rightarrow$ i):} By \eqref{Wigner-Formal-Adjoint} and H\"{o}lder's inequality for Schatten class operators, \eqref{HolderSchatten}, we have
\[
\left|\int G\overline{\F_W(T)}d\mu\right|=|\mathrm{tr}((\mathcal{E}_W(G))T^*)|\leq \|T\|_{\cS^p}\|\E_W(G)\|_{\cS^{p'}}\leq C\|T\|_{\cS^p}\|G\|_{L^{q'}(\mu)}.
\]
We therefore have
\[
\|\F_W(T)\|_{L^q(\mu)}=\sup_{\|G\|_{L^{q'}(\mu)}=1}\left|\int G\overline{\F_W(T)}d\mu\right|\leq C\|T\|_{\cS^p}.
\]

\end{proof}

\begin{remark}
When $\mu$ is the Lebesgue measure on $\R^{2d}$, then $\E_W(G)$ is the integrated Schr\"{o}dinger representation on the reduced Heisenberg group. In this case, $\E_W(G)$ is a compact operator on $L^2(\R^d)$ by \cite[Thm. 1.30]{Folland} whenever $G\in L^1(\R^{2d})$. A short proof of this result follows from Proposition \ref{Res-to-Ext-quantum}: If  $1<p\leq 2$ and $g\in L^p(\R^{2d})$, then by Proposition \ref{Operator Hausdorff-Young} it follows that
\[
\E_W(G)=\iint_{\R^{2d}} G(z)\rho(z)dz\in\cS^{p'}\subset\mathcal{K}.
\]
If $G\in L^1(\R^{2d})$, choose a sequence $\{G_n\}_{n\in\mathbb{N}}\subset \mathscr{S}(\R^{2d})$ such that $\|G_n-G\|_1\to 0$. Thus, Proposition \ref{Res-to-Ext-quantum} gives
\[
\|\E_W G-\E_W G_n\|_{L^2\to L^2}\leq \|G-G_n\|_{L^1(\R^{2d})}\xrightarrow{n\to\infty} 0,
\]
which shows that $\E_WG$ is the norm limit of compact operators, and is therefore compact.
\end{remark}
\section{Weyl quantization of measures}

In this section we will consider different quantization schemes whenever the symbol is the Fourier transform of a compactly supported Radon measure. We will also show a connection between the Weyl quantization of classical extension operators and the Fourier-Wigner extension operator discussed in the previous section.

Given proposition \ref{FollandProp}, one can investigate the Weyl-quantization of the classical extension operator $\mathcal{E}_\sigma\Psi$ for $\Psi\in L^1(\mu)$.

\begin{theorem}
Let $\mu$ be a finite Borel measure on $\R^{2d}$, and let $\Psi\in L^1(\mu)$. If we consider the symplectic extension operator
\begin{equation*}
    \E_\sigma\Psi(\zeta)=\mathcal{F}_\sigma(\Psi d\mu)(\zeta)=\int \Psi(z)e^{-2\pi i \sigma(\zeta,z)}d\mu(z),
\end{equation*}
then the Weyl-quantization of $\E_\sigma\Psi$ is given by
\begin{equation*}
    L_{\E_\sigma\Psi}=\E_W\Psi=\int \Psi(z)\rho(z)d\mu(z).
\end{equation*}
\end{theorem}
\begin{proof}
It follows from \cite[Thm. 11.2.5]{Grochenig} that $W(f,g)\in\mathscr{S}(\R^{2d})$ whenever $f,g\in \mathscr{S}(\R^d)$. By Parseval's identity for finite measures, Lemma \ref{Plancherelfor finite measures}, we have
\begin{align*}
    \iint_{\R^{2d}}\E_\sigma\Psi(x,\xi)W(f,g)(x,\xi)dxd\xi
    =&\iint_{\R^{2d}}\mathcal{F}_\sigma(\Psi d\mu)(x,\xi)\overline{W(g,f)(x,\xi)}dxd\xi\\
    =&\int\Psi(y,\eta)\overline{\mathcal{F}_\sigma(W(g,f))(y,\eta)}d\mu(y,\eta)\\
    =&\int\Psi(y,\eta)\overline{\langle g,\rho(y,\eta) f\rangle} d\mu(y,\eta)\\
    =&\left\langle\E_W\Psi f,g\right\rangle,
\end{align*}
where we used Lemma \ref{FourierOfWigner}, namely that the symplectic Fourier transform of the cross-Wigner distribution is the ambiguity function.

\end{proof}

This shows that the Weyl-quantization of the classical extension operator is given by the Fourier-Wigner extension operator.

Weyl quantization may be viewed as a $\tau$-quantization, and in the following we consider the extension operators in these quantization schemes. We are following the convention used in \cite{Luef}. For $\tau\in[0,1]$, the cross-$\tau$-Wigner distribution $W_\tau$ is defined as
\[
W_\tau(\varphi,\psi)(x,\omega)=\int_{\R^{d}}\varphi(x+\tau t)\overline{\psi(x+(1-\tau)t)}e^{-2\pi i t\cdot \omega}dt.
\]
It follows from \cite[Prop. 1.3.27]{CorderoBook} that
\[
W_{\tau}(\varphi,\psi)=\sigma_\tau*W(\varphi,\psi),
\]
where $W(\varphi,\psi)=W_{\frac{1}{2}}(\varphi,\psi)$ is the standard cross-Wigner distribution, and  $\sigma_\tau$ is given by
\begin{equation}\label{tau-symbol}
\sigma_\tau(x,\xi)=\begin{cases}\frac{2^d}{|2\tau-1|^d}e^{2\pi i \frac{2}{2\tau-1}x\cdot\xi}\quad &\tau\neq \frac{1}{2},\\
\delta_0,\quad &\tau=\frac{1}{2}.\end{cases}
\end{equation}
For any $a\in \mathscr{S}'(\R^{2d})$, we define the $\tau$-quantization as the operator $L_a^\tau$ such that
\[
\langle L_a^\tau \varphi,\psi\rangle=\langle a,W_\tau(\psi,\varphi)\rangle,
\]
holds for all $\varphi,\psi\in\mathscr{S}(\R^d)$.
It follows that
\[
\langle L_a^\tau \varphi,\psi\rangle=\langle a, W_\tau(\psi,\varphi)\rangle=\langle a,\sigma_\tau*W(\psi,\varphi)\rangle=\langle a*\sigma_{1-\tau}, W(\psi,\varphi)\rangle=\langle L_{a*\sigma_{1-\tau}} \varphi,\psi\rangle,
\]
for all $\varphi,\psi\in\mathscr{S}(\R^d)$. This shows that the $\tau$-quantization is given by convolving the symbol $a$ with $\sigma_{1-\tau}=\overline{\sigma_\tau}$ and then taking the standard Weyl quantization.

In particular, the $\tau$-quantization of the extension operator becomes
\[
L^\tau_{\E_\sigma G}=L_{\F_\sigma(Gd\mu)*\sigma_{1-\tau}}=L_{\F_{\sigma}(\F_\sigma(\sigma_{1-\tau})Gd\mu)}=\int \F_{\sigma}(\sigma_{1-\tau})(z)G(z)\rho(z)d\mu(z),
\]
where $\F_\sigma(\sigma_{1-\tau})$ denotes the distributional symplectic Fourier transform of $\sigma_{1-\tau}$. In the proof of \cite[Prop. 1.3.27]{CorderoBook} it is shown that
\[
\mathcal{F}^{-1}\sigma_{\tau}(y,\eta)=e^{-\pi i(2\tau-1)y\cdot \eta}.
\]
This implies that
\begin{equation}\label{Fourier-tau-symbol}
    \mathcal{F}_\sigma(\sigma_\tau)(y,\eta)=e^{\pi i(2\tau-1)y\cdot \eta},
\end{equation}
and thus the $\tau$-quantization of the extension operator becomes
\begin{equation}\label{Tau-Extension-Operator-Def}
L^\tau_{\E_\sigma G}=\int e^{-\pi i (2\tau-1)x\cdot \xi}G(x,\xi)\rho(x,\xi)d\mu(x,\xi).
\end{equation}

The proof of Theorem \ref{Char-measure-compact} utilise the interplay between localization operators and Weyl quantization. Recall that the localization operator with symbol $a\in\mathscr{S}'(\R^{2d})$ and windows $\varphi,\psi\in\mathscr{S}(\R^{d})$ is defined by the operator convolution
\[
\mathcal{A}_a^{\varphi,\psi}=a\star(\psi\otimes \varphi),
\]
and can be written in terms of a Weyl quantization with a smoothed Weyl symbol:
\[
\mathcal{A}_a^{\varphi,\psi}=L_{a*W(\psi,\varphi)}.
\]
Before proving Theorem $\ref{Char-measure-compact}$, the following lemma on compactly supported measures is needed.
\begin{lemma}\label{lem:ControlFourMeas}
    Let $1\leq p\leq \infty$ and $\mu$ be a compactly supported Radon measure such that $\F_\sigma(\mu)\in L^p(\R^{2d})$. Fix $z_0\in \mathrm{supp }{(\mu)}\subset \R^{2d}$ and define the measure $\nu$ such that
    \[
    d\nu=e^{-\pi \frac{|z-z_0|^2}{2}}d\mu.
    \]
    Then there exist $C(z_0,\mu)\geq 1$ such that
    \[
    \frac{1}{C}\|\F_\sigma(\nu)\|_{L^p}\leq \|\F_\sigma(\mu)\|_{L^p}\leq C\|\F_\sigma(\nu)\|_{L^p}.
    \]
\end{lemma}
\begin{proof}
    Define the function
    \[
    G(z)=e^{-\pi\frac{|z-z_0|^2}{2}}\in\mathscr{S}(\R^{2d}),
    \]
    such that $d\nu=Gd\mu$. It then follows that $\F_\sigma(\nu)=\F_\sigma(G)*\F_\sigma(\mu)$, and thus Young's inequality gives
    \[
    \|\F_\sigma(\nu)\|_{L^{p}}\leq \|\F_\sigma(G)\|_{L^1}\|\F_\sigma(\mu)\|_{L^p}.
    \]

    As $\mu$ is compactly supported, there exists $R>0$ such that $\mathrm{supp }(\mu)\subseteq B(z_0,R)$. Consider a smooth bump function $\Psi\in C_c^\infty(\R^{2d})$  such that $\Psi\equiv 1$ on $B(z_0,R)$ and $\Psi\equiv 0$ on $B(z_0,2R)$. Since $G$ is non-vanishing and smooth, it follows that $1/G$ is a $C^\infty$ function on any compact subset of $\R^{2d}$. Thus, we can define $\Phi\in C_c^\infty(\R^{2d})$ by 
    \[\Phi(z):=\frac{\Psi(z)}{G(z)}.\]
    Moreover, as $\Psi\equiv 1$ on the support of $\mu$, it follows that
    \[
    d\mu=e^{\pi\frac{|z-z_0|^2}{2}}d\nu=\Phi\,d\nu.
    \]
    The Fourier transform can then be written as $\F_\sigma(\mu)=\F_\sigma(\Phi)*\F_\sigma(\nu)$ and Young's inequality gives
    \[
    \|\F_\sigma(\mu)\|_{L^{p}}\leq \|\F_\sigma(\Phi)\|_{L^1}\|\F_\sigma(\nu)\|_{L^p}.
    \]
    The result now follows by choosing $C=\max\{\|\F_\sigma(G)\|_{L^1},\|\F_\sigma(\Phi)\|_{L^1}\}$.
\end{proof}

\begin{proof}[Proof of Theorem \ref{Char-measure-compact}]
Assume first that $\mu$ is a Rajchman measure. Recall from \eqref{WignerGaussian} that $W(g_0,g_0)(z)=2^d\exp{(-2\pi|z|^2)}$, and consider the measure $\nu$ given by
\begin{equation}\label{nu-def}
d\nu=\frac{1}{\mathcal{F}_{\sigma}(W(g_0,g_0))}d\mu=e^{\pi\frac{|z|^2}{2}}d\mu,
\end{equation}
which is a finite measure as $\text{supp}(\mu)$ is compact.
Moreover,
\[\F_\sigma(\nu)*W(g_0,g_0)=\F_\sigma(\F_\sigma(W(g_0,g_0))d\nu)=\F_\sigma(\mu)\in C_0(\R^{2d}),
\]
as $\mu$ is a compactly supported Rajchman measure. This shows that $\mathcal{A}_{\F_\sigma(\nu)}^{g_0,g_0}=L_{\F_{\sigma}(\mu)}$, and thus the Weyl quantization $L_{\F_{\sigma}(\mu)}$ is a compact operator by Theorem \ref{TauberianTheorem} as $\F_\sigma(W(g_0,g_0))(z)=\exp{(-\pi|z|^2/2)}> 0$ for all $z\in \R^{2d}$.

On the other hand, if we assume $L_{\F_\sigma(\mu)} $ is compact, then $\mathcal{A}^{g_0,g_0}_{\F_\sigma(\nu)}$ is a compact operator where the measure $\nu$ is given by \eqref{nu-def}. Thus, by Theorem \ref{TauberianTheorem}
\[
\F_\sigma(\nu)*e^{-\pi|\cdot|^2}\in C_0(\R^{2d}).
\]
However, as the Fourier transform of the Gaussian has no zeros it follows from the classical Tauberian theorem that
\[
\F_\sigma(\nu)*\varphi\in C_0(\R^{2d})
\]
for any $\varphi\in L^1(\R^{2d})$. In particular, choosing $\varphi=\F_\sigma(W(g_0,g_0))$ gives
\[
\mathcal{F}_\sigma(\mu)=\F_\sigma(\F_\sigma(W(g_0,g_0)d\nu)=\F_\sigma(\nu)*W(g_0,g_0)\in C_0(\R^{2d}),
\]
which shows that $\mu$ is a Rajchman measure.

Since $\mu$ is a compactly supported Radon measure, it follows that there exist $z_0\in\R^{2d}$ and $0<R<\infty$ such that $\text{supp}(\mu)\subset B(z_0,R)$. 

Let $T\in \cS^1$, and consider the Cohen's class distribution $Q_T(g_0,g_{z_0})=T\star (g_{z_0}\otimes g_{0})\in L^1(\R^{2d})$ and the measure $\nu$ given by $d\mu=\exp(-\pi|z-z_0|^2/2)d\nu$.
Then by Proposition \ref{Duality of restriction operators} and Lemma \ref{Plancherelfor finite measures},
\begin{align*}
\left|\langle T, L_{\F_{\sigma}(\mu)}\rangle\right|
=\left|\langle T, \E_W(1)\rangle\right|
=\left|\int \F_W(T)d\mu\right|
=&\,\left|\int F_W(T)(z)e^{-\pi\frac{|z-z_0|^2}{2}}d\nu(z)\right|\\
=&\,\left|\int e^{-\pi i\sigma(z_0,z)}\F_\sigma(T\star(g_{z_0}\otimes g_{0}))(z)d\nu(z)\right|\\
=&\,\left|\int (T\star(g_{z_0}\otimes g_{0}))(\zeta+\tfrac{z_0}{2})\F_\sigma(\nu)(-\zeta)d\zeta\right|.
\end{align*}
By the triangle inequality and H\"{o}lder's inequality, followed by Werner-Young's inequality, we have
\[
\left|\langle T,L_{\F_\sigma(\mu)}\rangle\right|\leq \|T\star(g_0\otimes g_{z_0})\|_{L^{p}}\|\F_{\sigma}(\nu)\|_{L^{p'}}\leq C \|g_0\otimes g_{z_0}\|_{\cS^{1}}\|\F_{\sigma}(\mu)\|_{L^{p'}}\|T\|_{\cS^p},
\]
where we used Lemma \ref{lem:ControlFourMeas} to control $\|\F_\sigma(\nu)\|_{p'}$ in terms of $\|\F_\sigma(\mu)\|_{p'}$, which is finite as $\F_{\sigma}(\mu)\in L^{p'}(\R^{2d})$ by assumption. Since $\cS^1$ is dense in $\cS^{p}$ for $p<\infty$, we can extend the estimate to any $T\in \cS^p$. Thus, we have
\[
\|L_{\F_\sigma(\mu)}\|_{\cS^{p'}}=\sup_{\|T\|_{\cS^p}=1}\left|\langle T,L_{\F_\sigma(\mu)}\rangle\right|\leq C \|g_0\otimes g_{z_0}\|_{\cS^{1}}\|\F_{\sigma}(\mu)\|_{L^{p'}}<\infty,
\]
which shows that $L_{\F_{\sigma}(\mu)}\in\cS^{p'}$ whenever $\F_\sigma(\mu)\in L^{p'}(\R^{2d})$.

Assume $L_{\F_\sigma(\mu)}\in \cS^{p'}$ and consider the measure $d\nu=\exp{(-\pi|z-z_0|^2/2)}d\mu$, which is a compactly supported finite measure as $\mu$ is compactly supported. For any $\Psi\in \mathscr{S}(\R^{2d})$, let $T_{z_0}\Psi(z)=\Psi(z-z_0)$ be the translation by $z_0$. Then it follows that 
\begin{align*}
\left|\langle \F_\sigma(\nu),\Psi\rangle\right|
=\left|\int \F_\sigma(\Psi)\,d\nu\right|=\left|\int e^{\pi i\sigma(z_0,z)}\F_\sigma\left(T_{-\frac{z_0}{2}}\Psi\right)\,d\nu\right|
=&\,\left|\int e^{\pi i\sigma(z_0,z)}\F_\sigma\left(T_{-\frac{z_0}{2}}\Psi\right)e^{-\pi\frac{|z-z_0|^2}{2}}\,d\mu\right|\\
=&\left|\int \F_W\left(\left(T_{-\frac{z_0}{2}}\Psi\right)\star(g_{z_0}\otimes g_{0})\right)\,d\mu\right|\\
=&\,\left|\left\langle \left(T_{-\frac{z_0}{2}}\Psi\right)\star (g_{z_0}\otimes g_{0}),\mathcal{E}_W(1)\right\rangle\right|\\
\leq&\, \left\|\left(T_{-\frac{z_0}{2}}\Psi\right)\star(g_{z_0}\otimes g_{0})\right\|_{\cS^p}\|L_{\F_\sigma(\mu)}\|_{\cS^{p'}},
\end{align*}
where we used $\F_\sigma(\Psi)=\F_\sigma(T_{z_0}T_{-z_0}\Psi)=\exp(2\pi i\sigma(z_0,\cdot))\F_\sigma(T_{-z_0}\Psi)$ for the second equality. Another application of Werner-Young's inequality yields
\[
\left|\langle \F_\sigma(\nu),\Psi\rangle\right|\leq \left\|L_{\F_\sigma(\mu)}\right\|_{\cS^{p'}}\left\|g_{z_0}\otimes g_0\right\|_{\cS^1}\left\|T_{-\frac{z_0}{2}}\Psi\right\|_{L^p(\R^{2d})}=
\|L_{\F_\sigma(\mu)}\|_{\cS^{p'}}\|\Psi\|_{L^p(\R^{2d})},
\]
as the Lebesgue measure is translation invariant, and $g_0$ and $g_{z_0}$ are $L^2$-normalised. By the density of $\mathscr{S}(\R^{2d})$ for $1\leq p<\infty$ and Lemma \ref{lem:ControlFourMeas}, 
\[
\|\F_\sigma(\mu)\|_{L^{p'}(\R^{2d})}
\leq C\|\F_\sigma(\nu)\|_{L^{p'}(\R^{2d})}\leq C\|L_{\F_\sigma(\mu)}\|_{\cS^{p'}},
\]
for all $1<p'\leq \infty$.

For $p'=1$ we note that $C_0(\R^{2d})$ is the $L^\infty$ closure of $\mathscr{S}(\R^{2d})$, and thus the above estimate extends to all $\Psi\in C_0(\R^{2d})$. This means that $\F_\sigma(\mu)$ is a bounded linear functional on $C_0(\R^{2d})$, and by the Riesz representation theorem we can identify $\F_\sigma(\mu)$ by a Radon measure $\nu$ for which we can bound the total variation by
\[
\|\nu\|\leq C\|L_{\F_\sigma(\mu)}\|_{\cS^{1}}.
\]
Since $\mu$ is a compactly supported Radon measure, it follows that $\F_\sigma(\mu)\in C^\infty(\R^{2d})$. This implies that
\[
\|\F_\sigma(\mu)\|_{L^1(\R^{2d})}=\|\nu\|\leq C\|L_{\F_\sigma(\mu)}\|_{\cS^{1}},
\]
which then completes the proof.
\end{proof}

\begin{remark}
    Hambrook and \L aba showed that if $\mu$ is supported on a set of Hausdorff dimension $0<\alpha<2d$, then $\|\F_\sigma(\mu)\|_{L^p}=\infty$ for $p<4d/\alpha$ in \cite{Hambrook-Laba}. In particular, we must have $L_{\F_\sigma(\mu)}\notin \cS^{p}$ if $\mu$ is supported on a set of Hausdorff dimension $\alpha$ and $p<4d/\alpha$. If $\alpha=2d/n$ for some integer $n\geq2$, then Chen and Seeger  proved the existence of probability measures supported on a set of Hausdorff dimension $\alpha$ for which Fourier restriction holds for $p=4d/\alpha$ \cite{Chen-Seeger}. In this case $L_{\F_\sigma(\mu)}\in \cS^\frac{4d}{\alpha}$ by Corollary \ref{equi-Cor}.
\end{remark}

In order to extend the results to the $\tau$-quantization, we first need the following lemma.
\begin{lemma}\label{Bounded to vanishing lemma}
Let $f\in C_b(\R^{d})$ be uniformly continuous. Then $f*e^{-\pi|\cdot|^2}\in C_0(\R^d)$ if and only if $f\in C_0(\R^d)$.
\end{lemma}
\begin{proof}
    Assume first that $f*e^{-\pi|\cdot|^2}\in C_0(\R^d)$. By Theorem \ref{TauberianTheorem} it follows that $f*g\in C_0(\R^d)$ for all $g\in L^1(\R^d)$. 
    For any $r>0$ consider the function
    \[
    g_r(x)=\frac{\chi_{B(0,r)}(x)}{|B(0,r)|}\in L^1 (\R^{d}).
    \]
    Then
    \begin{equation*}
        f*g_r(x)=\fint_{B(0,r)}f(x-y)dy=\fint_{B(x,r)}f(y)dy=A_rf(x)\in C_0(\R^{d}).
    \end{equation*}
    The function $A_rf(x)$ is jointly continuous in $r$ and $x$ and tends to $f(x)$ as $r\to 0$ for a.e. $x$ by the Lebesgue differentiation theorem. In fact, $A_rf\to f$ pointwise for any $x$ as $f$ is continuous.

    Since $f$ is uniformly continuous, we have that for every $\varepsilon>0$ there exists $r_\varepsilon>0$ such that $|f(x)-f(y)|<\varepsilon$ when $|x-y|<r$ for any $r<r_\varepsilon$. Thus, for any $x\in \R^{d}$ we have
    \[
    |f(x)-A_rf(x)|\leq \fint_{B(x,r)}|f(x)-f(y)|dy<\varepsilon,
    \]
    whenever $r<r_\varepsilon$. This shows that $A_rf\to f$ uniformly as $r\to 0$, and so $f\in C_0(\R^d)$ as $(C_0(\R^{d}),\|\cdot\|_\infty)$ is complete.

    Assume now that $f\in C_0(\R^d)$. We note that for any $x,y\in\R^{d}$,
    \[
    |f(x-y)e^{-\pi|y|^2}|\leq \|f\|_{\infty}e^{-\pi|y|^2}\in L^1(\R^{d}).
    \]
    By Lebesgue's dominated convergence theorem
    \[
    \lim_{x\to \infty}f*e^{-\pi|\cdot|^2}(x)=\int_{\R^{d}}\lim_{x\to\infty} f(x-y)e^{-\pi|y|^2}dy=0,
    \]
    as $\lim_{x\to \infty}f(x)=0$. This shows that $f*e^{-\pi|\cdot|^2}\in C_0(\R^d)$ as the convolution is uniformly continuous.
\end{proof}
\begin{remark}
Note that the uniformly continuity condition on $f$ in Lemma \ref{Bounded to vanishing lemma} cannot be removed: This can be seen by considering $f(x)=\sin(\pi x^2)$. Then $\check{f}(\xi)=2^{-\frac{1}{2}}(\cos(\pi\xi^2)-\sin(\pi\xi^2))\in L^\infty(\R)$. This means that $\check{f}e^{-\pi|\cdot|^2}\in L^1(\R)$ and so $f*e^{-\pi|\cdot|^2}=\widehat{\Check{f}e^{-\pi|\cdot|^2}}\in C_0(\R)$ by the Riemann-Lebesgue lemma, but $f\notin C_0(\R)$.
\end{remark}
\begin{remark}
    Let $\mu$ be a compactly supported Radon measure on $\R^{2d}$ and let $R>0$ be such that $\text{supp}(\mu)\subset B(0,R)$. Then $\F_\sigma(\mu)\in C^{\infty}(\R^{2d})$, and $\|\partial^{\alpha}\F_\sigma(\mu)\|_\infty\leq(2\pi R)^{|\alpha|}\|\mu\|$ for all $\alpha\in \mathbb{N}^{2d}$. In particular, $\F_\sigma(\mu)$ is uniformly continuous and so Lemma \ref{Bounded to vanishing lemma} is applicable to $\F_\sigma(\mu)$.
\end{remark}

A version of Theorem \ref{Char-measure-compact} can also be established for the $\tau$-Weyl quantization. The $\tau$-Weyl quantization of a symbol $a$ is equal to the Weyl quantization of the symbol $a*\sigma_{1-\tau}$, where $\sigma_\tau\in C^{\infty}(\R^{2d})$ whenever $\tau\neq 1/2$. By exploiting this connection, we obtain the following result.
\begin{theorem}\label{tau-Weyl-estimates}
Let $\mu$ be a compactly supported Radon measure on $\R^{2d}$, and let $0\leq \tau\leq 1$. Then the following holds:
\begin{enumerate}[a)]
    \item The operator
    \begin{equation*}
        L^{\tau}_{\F_\sigma(\mu)}=\int e^{-\pi i (2\tau-1)x\cdot \xi}\rho(x,\xi)d\mu(z,\xi),
    \end{equation*}
    is compact if and only if $\F_\sigma(\mu)\in C_0(\R^{2d})$.
    \item If there exists $0<\beta\leq \alpha<2d$ and constants $C_\alpha,C_\beta>0$ such that
    \begin{align*}
        \mu(B(z,r))\leq&\, C_\alpha r^\alpha\, \text{for all }\, z\in \R^{2d}\, \text{and }\, r>0,\\
        |\F_\sigma(\mu)(\zeta)|\leq&\, C_\beta (1+|\zeta|)^{-\frac{\beta}{2}}\, \text{for all }\,\zeta\in\R^{2d}.
    \end{align*}
    Then for all
    \begin{equation*}
        p\geq \frac{2(4d-2\alpha+\beta)}{\beta},
    \end{equation*}
    there exists $C=C(p)>0$ such that
    \begin{equation*}
        \|L^\tau_{\F_\sigma(gd\mu)}\|_{\cS^{p}}\leq C\|g\|_{L^{2}(\mu)},
    \end{equation*}
    for all $g\in L^2(\mu)$.

    \item If $\mu$ is the surface measure of a compact hypersurface $\mathcal{M}\subset\R^{2d}$ with non-vanishing Gaussian curvature, then $L_{\F_\sigma(\mu)}^\tau\in \cS^{p}$ for $p> 4d/(2d-1)$.
\end{enumerate}
\end{theorem}
\begin{proof}
    Since $L_{\F_\sigma(\mu)}^\tau=L_{\F_\sigma(\mu)*\sigma_{1-\tau}}$, where $\sigma_{1-\tau}$ is defined as in \eqref{tau-symbol}, it follows from Theorem \ref{Char-measure-compact} that $L_{\F_\sigma(\mu)}^\tau$ is compact if and only if $\F_\sigma(\mu)*\sigma_{1-\tau}\in C_0(\R^{2d})$. Let $G_0$ denote the $L^2$-normalised Gaussian on $\R^{2d}$.
    
    Assume first that $\F_\sigma(\mu)\in C_0(\R^{2d})$. Since $\mu$ is a compactly supported measure, we have that the measure
    \[
    d\nu=\F_\sigma(\sigma_{1-\tau})d\mu,
    \]
    is also a compactly supported finite measure as $|\F_\sigma(\sigma_{1-\tau})|=1$ by \eqref{Fourier-tau-symbol}, and thus
\[\F_\sigma(\mu)*\sigma_{1-\tau}=\F_\sigma(\F_\sigma(\sigma_{1-\tau})d\mu)=\F_\sigma(\nu)\in C_b(\R^{2d}).
    \]

    It follows from Lemma \ref{Bounded to vanishing lemma} that $\F_\sigma(\mu)*G_0\in C_0(\R^{2d})$ as $\F_\sigma(\mu)\in C_0(\R^{2d})$. In fact, by Wiener's Tauberian Theorem, Theorem \ref{TauberianTheorem}, we know that $\F_\sigma(\mu)*\Phi\in C_0(\R^{2d})$ for any $\Phi\in L^1(\R^{2d})$. Moreover, as
    $\F_\sigma(\sigma_{1-\tau})\in C^\infty(\R^{2d})$ and each derivative has at most polynomial growth, it follows that
    \[
    \sigma_{1-\tau}*G_0=\F_\sigma(\F_\sigma(\sigma_{1-\tau})G_0)\in\mathscr{S}(\R^{2d})\subset L^1(\R^{2d}),
    \]
    as $G_0\in \mathscr{S}(\R^{2d})$ and is invariant under the symplectic Fourier transform. In particular
    \[
    (\F_\sigma(\mu)*\sigma_{1-\tau})*G_0=\F_\sigma(\mu)*\F_\sigma(\F_\sigma(\sigma_{1-\tau})G_0)\in C_0(\R^{2d}),
    \]
     and since $\F_\sigma(\sigma_{1-\tau})d\mu$ is a compactly supported measure, it follows from Lemma \ref{Bounded to vanishing lemma} that \[\F_\sigma(\F_\sigma(\sigma_{1-\tau}d\mu))=\F_\sigma(\mu)*\sigma_{1-\tau}\in C_0(\R^{2d}).\] 
    Theorem \ref{Char-measure-compact} then gives that $L_{\F_\sigma(\mu)}^\tau$ is a compact operator.

    If $L_{\F_\sigma(\mu)}^\tau$ is compact, then $\F_\sigma(\mu)*\sigma_{1-\tau}\in C_0(\R^{2d})$ by Theorem \ref{Char-measure-compact}. This implies that
    \[(\F_\sigma(\mu)*\sigma_{1-\tau})*g_0\in C_0(\R^{2d}),\]
    by Lemma \ref{Bounded to vanishing lemma}. The same conclusion also holds whenever $G_0$ is replaced by any $\Phi\in L^1(\R^{2d})$ by Wiener's Tauberian Theorem. Note that $\F_\sigma(\sigma_{1-\tau})^{-1}=\overline{\F_\sigma(\sigma_{1-\tau})}$
    and since $\overline{\F_\sigma(\sigma_{1-\tau})}G_0\in\mathscr{S}(\R^{2d})$
    it follows that
    \[
    \F_\sigma(\mu)*G_0=\F_\sigma(\mu)*\F_\sigma\left(\F_\sigma(\sigma_{1-\tau})\overline{\F_\sigma(\sigma_{1-\tau})}G_0\right)=(\F_\sigma(\mu)*\sigma_{1-\tau})*\F_\sigma(\overline{\F_\sigma(\sigma_{1-\tau})}G_0)\in C_0(\R^{2d}).
    \]
    Lemma \ref{Bounded to vanishing lemma} therefore gives us that $\F_\sigma(\mu)\in C_0(\R^{2d})$. This concludes part $a)$ of the theorem.

    For part $\mathrm{b})$ recall that 
    \[
    L_{\F_\sigma(Gd\mu)}^\tau=\E_W(\F_\sigma(\sigma_{1-\tau})G).
    \]
    Since $|\F_\sigma(\sigma_{1-\tau})|=1$, it follows that
    \[
    \F_\sigma(\sigma_{1-\tau})G\in L^2(\mu),
    \]
    whenever $G\in L^2(\mu)$. The result then follows from Theorem \ref{QuantumBakTheorem}.

    For part $\mathrm{c})$ it is enough to consider the measure $d\nu=\F_\sigma(\sigma_{1-\tau})d\mu$. Then $|\F_\sigma(\nu)(\zeta)|\leq C(1+ |\zeta|)^{-\frac{2d-1}{2}}$ by the method of stationary phase, see for example \cite[Cor. 4.16]{Schlag}. The second part of Theorem \ref{Char-measure-compact} gives
    \[
    L_{\F_\sigma(\mu)}^\tau=L_{\F_\sigma(\nu)}\in \cS^{p},
    \]
    for $p> 4d/(2d-1)$.
\end{proof}


\section*{Acknowledgement}
The authors would like to thank Sigrid Grepstad and Eugenia Malinnikova for their helpful comments on earlier versions of this manuscript. The authors also appreciate the feedback given by  Robert Fulchse and an anonymous referee which greatly helped to improve the manuscript.

\printbibliography

\end{document}